\documentclass[11pt, a4paper]{amsart}
\usepackage{a4wide}
\usepackage{graphicx,enumitem}
\usepackage{algorithm,algorithmic}
\usepackage[english]{babel}
\usepackage{subfigure}
\usepackage{units}
\usepackage{diagbox}
\usepackage{color}
\usepackage{caption}

\usepackage{float}
\newfloat{algorithm}{t}{lop}
\usepackage{ulem}
\usepackage{cancel}

 \usepackage{amsfonts}
 \usepackage{amssymb}
\usepackage{latexsym}

\usepackage[pdftex,
            pdftitle={Rapid covariance based sampling of linear SPDE in MLMC},
            pdfauthor={A. Petersson},
            bookmarksopen,
            colorlinks,
            linkcolor=black,
            urlcolor=black,
	    citecolor=black
]{hyperref}

\usepackage{dsfont}

\allowdisplaybreaks
\usepackage{mathtools}
\DeclarePairedDelimiter{\ceil}{\lceil}{\rceil}
\DeclarePairedDelimiter{\floor}{\lfloor}{\rfloor}

\newcommand{\R}{\mathbb{R}}

\newcommand{\N}{\mathbb{N}}

\newcommand{\cA}{\mathcal{A}}
\newcommand{\cB}{\mathcal{B}}

\newcommand{\cF}{\mathcal{F}}

\newcommand{\cL}{\mathcal{L}}

\newcommand{\Op}{\operatorname{O}}

\DeclareMathOperator{\E}{\mathbb{E}} %expectation
\DeclareMathOperator{\Var}{\mathsf{Var}} %variance
\DeclareMathOperator{\Cov}{\mathsf{Cov}}
\DeclareMathOperator{\vect}{\mathsf{Vec}}

\DeclareMathOperator{\trace}{Tr}
\newcommand{\cLnull}{{\cL_2^0}}

\newcommand{\dd}{\, \mathrm{d}}
\newcommand{\dom}{\mathrm{dom}}

\newcommand{\inpro}[3][{}]{ \left\langle #2 , #3 \right\rangle_{#1} }

\newcommand{\norm}[2][{}]{\| #2 \|_{#1}}

\newtheorem{lemma}{Lemma}[section]
\newtheorem{proposition}[lemma]{Proposition}

\newtheorem{theorem}[lemma]{Theorem}

\theoremstyle{remark}
\newtheorem{remark}[lemma]{Remark}

\theoremstyle{definition}

\newtheorem{assumption}[lemma]{Assumption}
\newtheorem{example}[lemma]{Example}

\begin{document}
\title[Rapid covariance based sampling of linear SPDE in MLMC]{
Rapid covariance-based sampling of linear SPDE approximations in the multilevel Monte Carlo method  
}

\author[A.~Petersson]{Andreas Petersson} \address[Andreas Petersson]{\newline Department of Mathematical Sciences
\newline Chalmers University of Technology \& University of Gothenburg
\newline S--412 96 G\"oteborg, Sweden.} 
\email[]{andreas.petersson@chalmers.se}

\thanks{
Acknowledgement.
	The author wishes to express many thanks to Annika Lang, Stig Larsson for support and fruitful discussions, to Michael B. Giles for helpful comments and to three anonymous referees who helped to improve the results and the presentation. The work was supported in part by the Swedish Research Council under Reg.~No.~621-2014-3995 and the Knut and Alice Wallenberg foundation.
}

\date{July 18, 2019}

\begin{abstract}
The efficient simulation of the mean value of a non-linear functional of the solution to a linear stochastic partial differential equation (SPDE) with additive Gaussian noise is considered. A Galerkin finite element method is employed along with an implicit Euler scheme to arrive at a fully discrete approximation of the mild solution to the equation. A scheme is presented to compute the covariance of this approximation, which allows for rapid sampling in a Monte Carlo method. This is then extended to a multilevel Monte Carlo method, for which a scheme to compute the cross-covariance between the approximations at different levels is presented. In contrast to traditional path-based methods it is not assumed that the Galerkin subspaces at these levels are nested. The computational complexities of the presented schemes are compared to traditional methods and simulations confirm that, under suitable assumptions, the costs of the new schemes are significantly lower.
\end{abstract}

\maketitle

\section{Introduction}%\label{sec:intro}

Stochastic partial differential equations (SPDE) have many applications in engineering, finance, biology and meteorology. These include filtering problems, pricing of energy derivative contracts and modeling of sea surface temperature. For an overview of applications, we refer to \cite{DPZ14,LR17}. A natural quantity of interest for an SPDE is the expected value of a non-linear functional of the solution of the equation at a fixed time. This includes moments of the solution but also more concrete quantities, such as, in the case that the SPDE models sea surface temperature, the average amount of area in which the temperature exceeds a given temperature distribution. In order to determine such quantities, numerical approximations of the SPDE have to be considered, since analytical solutions are in general unavailable. 

The field of numerical analysis of SPDE is very active and a multitude of approximations have been considered in the literature, see e.g., \cite{JK09} and \cite[Section 10.9]{LPS14} for an overview. In this paper we take the approach of \cite{K14}, where the author considers an SPDE of evolutionary type and employs a Galerkin method for the spatial discretization of the equation (which includes both spectral and finite element methods) along with a drift-implicit Euler--Maruyama scheme for the temporal discretization. The finite element method in particular is useful and flexible as no explicit knowledge of eigenfunctions or eigenvalues is needed. The author of \cite{K14} does, however, omit the problem of how to, given this approximation, efficiently estimate expected values. In this paper we formulate methods for this problem that, under suitable assumptions, outperform standard methods based on the discretization considered in~\cite{K14}.

Typically, the approximation of expected values is accomplished by a \textit{Monte Carlo} method (MC), i.e., by computing a large number of sample paths of the approximate solution and taking the average of the functional of interest applied to each path. This is however quite expensive. Starting with the publication of~\cite{G08}, the \textit{multilevel Monte Carlo method} (MLMC) has become popular, since it can reduce computational cost while retaining accuracy. The method was first considered in~\cite{H98} for the evaluation of functionals arising from the solution of integral equations. We refer to \cite{G15} for an introduction to this active field and to \cite{BL12a,BLS13} for the first applications of MLMC to finite element approximations of SPDE. MLMC was first considered for SPDE in the thesis~\cite{GR08}, where a spectral Galerkin discretization was used.

Even though the MLMC method decreases the computational cost of the approximation of expected values, it is still fairly expensive. In this paper we formulate covariance-based variants of the MC and MLMC methods. The idea is to exploit the fact that as long as the considered SPDE has additive noise and is linear, then the approximation from \cite{K14} of the end-time solution is Gaussian. Since a Gaussian random variable is completely determined by its mean and covariance, calculating these parameters provides an efficient way of sampling the approximation (Algorithm~\ref{alg:cov_mc} below). To incorporate this idea in an MLMC method (Algorithm~\ref{alg:cov_mlmc}) we calculate the cross-covariance between two SPDE approximations in different Galerkin subspaces. In contrast to \cite{BL12a} and \cite{BLS13}, the subspace sequence is not assumed to be nested. We demonstrate, using theoretical computations and numerical simulations, that the computational costs of these new algorithms are, under mild assumptions, substantially lower than their traditional path-based alternatives (Algorithms~\ref{alg:path_mc} and~\ref{alg:path_mlmc} respectively).

The paper is organized as follows: In Section~\ref{sec:setting} we recapitulate the theoretical setting and approximation results of \cite{K14}. We also introduce the assumptions we make along with a stochastic advection-diffusion equation as a concrete example that fulfills these. In Section~\ref{sec:mc} we introduce a covariance-based method for computing samples of SPDE approximations in an MC setting and compare the complexity of it to the traditional path-based method. We extend this in Section~\ref{sec:mlmc} to the setting of the MLMC method. Section~\ref{sec:implementation} contains a description of the numerical implementation of our methods and a discussion of our assumptions. Finally in Section~\ref{sec:numerics} we demonstrate the efficiency of our approach by simulation of the stochastic heat equation. 

\section{Stochastic Partial Differential Equations and their Approximations}
\label{sec:setting}

Let $(H, \inpro{\cdot}{\cdot}, \norm{\cdot})$ be a real separable Hilbert space and let $-A \colon \dom(-A) \subset H \to H$ be a positive definite, self-adjoint operator with a compact inverse on $H$. For a fixed time $T<\infty$, let $(\Omega, \cA, (\cF_t)_{t \in [0,T]}, P)$ be a complete filtered probability space satisfying the usual conditions. In this context we consider the linear SPDE
\begin{equation}
\label{eq:spde}
\begin{split}
\dd X(t) &= \big(AX(t) + F(t,X(t)) \big) \dd t + G(t) 
\dd W(t), \\
X(0) &= x_0,
\end{split}
\end{equation}
for $t \in [0,T]$. Here $W$ is an $H$-valued cylindrical $Q$-Wiener process, $x_0$ is a random member of a subspace of $H$, while $F$ and $G$ are mappings that fulfill Assumption~\ref{assumptions:1} below. The solution $X = (X(t))_{t\in [0,T]}$ is then an $H$-valued stochastic process. 
Equation \eqref{eq:spde} is treated with the semigroup approach of \cite[Chapter 7]{DPZ14}, resulting in a so-called mild solution of the equation. In order to introduce this notion, we start by describing the spectral structure induced by $A$ on $H$. 

By the spectral theorem applied to $(-A)^{-1}$, there is an orthonormal eigenbasis $( e_i )_{i \in \N}$ of $H$ and a positive sequence $( \lambda_i )_{i \in \N}$ of eigenvalues of $-A$ that is increasing and for which $\lim_i \lambda_i = \infty$. For $r \in \R$, fractional powers of $-A$ are defined by $(-A)^\frac{r}{2} f = \sum^\infty_{i = 1} \lambda_i^\frac{r}{2} \inpro{f}{e_i} e_i$
for $f \in \dot{H}^r = \dom((-A)^\frac{r}{2})$, which is characterized by \begin{equation*}
	\dot{H}^r = \left\{ f = \sum_{i=1}^{\infty} f_i e_i : f_i \in \R \text{ for all } i \in \N \text{ and } \norm[r]{f}^2 = \sum_{i=1}^{\infty} \lambda_i^{r} f_i^2 < \infty \right\}.
\end{equation*}
This is a separable Hilbert space with the inner product
$\inpro[r]{\cdot}{\cdot} = \inpro{(-A)^\frac{r}{2} \cdot }{(-A)^\frac{r}{2} \cdot }$. For $r> 0$, we have the Gelfand triple $\dot{H}^{r} \subseteq H \subseteq \dot{H}^{-r}$ since $\dot{H}^{-r} \cong (\dot{H}^r)^*$, the dual of $\dot{H}^r$. The operator $A$ is the generator of an analytic semigroup $E = (E(t))_{t \ge 0} \subset \cL(H)$.

Next, we briefly recapitulate some notions from functional analysis and probability theory. For two real separable Hilbert spaces $H_1$ and $H_2$, we denote by $H_1 \oplus H_2 = \{ [f,   u]' : f \in H_1, u \in H_2\}$ the Hilbert (external) direct sum of $H_1$ and $H_2$, with an inner product defined by $\inpro[H_1 \oplus H_2]{[f_1, u_1]'}{[f_2, u_2]'} = \inpro[H_1]{f_1}{f_2} + \inpro[H_2]{u_1}{u_2}$, $f_1, f_2 \in H_1$, $u_1, u_2 \in H_2$. Similarly, by $H_1 \otimes H_2$ we denote the Hilbert tensor product, i.e., the completion of the algebraic tensor product of $H_1$ and $H_2$ under the norm induced by the inner product $\inpro[H_1 \otimes H_2]{f_1 \otimes u_1}{f_2 \otimes u_2} = \inpro[H_1]{f_1}{f_2}\inpro[H_2]{u_1}{u_2}$, $f_1, f_2 \in H_1$, $u_1, u_2 \in H_2$. For $H_1 = H_2$ we write $H_1^{\otimes 2} = H_1 \otimes H_1$ and $f^{\otimes 2} = f \otimes f$ for $f \in H_1$. We denote by $\cL(H_1,H_2)$ and $\cL_2(H_1,H_2)$, or $\cL(H)$ respectively $\cL_2(H)$ when $H_1 = H_2 = H$, the spaces of linear respectively Hilbert--Schmidt operators from $H_1$ to $H_2$.  Similarly, the family of \textit{trace class} operators on $H$ is denoted by $\cL_1^{\mathrm{s}}(H)$ and consists of those operators $K \in \cL(H)$ that are positive semidefinite and symmetric and for which the series $\trace{K} = \sum_{i=1}^\infty \inpro{K e_i}{e_i}$ is absolutely convergent.

For an $H$-valued random variable $X \in L^1(\Omega;H)$, i.e., $\E[\norm{X}] < \infty$, the \textit{expected value}, or \textit{mean}, of $X$ is defined by the Bochner integral $
\E[X] = \int_{\Omega} X(\omega) \dd P(\omega)$. If $X \in L^2(\Omega;H)$, we define the \textit{covariance} or \textit{covariance operator} of $X$ by 
\begin{equation*}
\Cov(X) = \E[(X - \E[X])^{\otimes 2}] = \E[X^{\otimes 2}] - \E[X]^{\otimes 2} \in H^{\otimes 2}.
\end{equation*} More generally, for Hilbert spaces $H_1$ and $H_2$, we define the \textit{cross-covariance} or \textit{cross-covariance operator} of $Y \in L^2(\Omega;H_2)$ and $Z \in L^2(\Omega;H_1)$ by 
\begin{equation*}
\Cov(Y,Z) = \E[(Y - \E[Y]) \otimes (Z - \E[Z])] = \E[Y \otimes Z] - \E[Y] \otimes \E[Z] \in H_2 \otimes H_1
\end{equation*} so that $\Cov(X) = \Cov(X,X).$
Calling these quantities operators is justified by the fact that $H_2 \otimes H_1 \simeq \cL_2(H_1,H_2) \subseteq \cL(H_1,H_2)$. The action of the cross-covariance is given by $\Cov(Y,Z) = \E[\inpro[H_1]{Z - \E[Z]}{\cdot}(Y - \E[Y])]$ which means that $\Cov(X) \in \cL_1^{\mathrm{s}}(H)$, implying that it has a unique square root. 

We next recall that an $H$-valued random variable $X$ is said to be Gaussian if $X \in H$ $P$-a.s. and $\inpro{X}{f}$ is a real-valued Gaussian random variable for all $f \in H$. In this case $X \in L^p(\Omega;H)$  for all $p \ge 1$ so $\Cov(X)$ is well-defined. Now, a stochastic process $W: [0,T] \times \Omega \to H$ is said to be an \textit{$H$-valued $Q$-Wiener process} adapted to $(\cF_t)_{t \in [0,T]}$ if $W(0)=0$, $W$ has $P$-a.s. continuous trajectories, and if there exists a self-adjoint trace class operator $Q \in \cL(H)$ such that for each $0 \le s < t \le T$, $W(t)-W(s)$ is Gaussian with zero mean and covariance $(t-s)Q$ and $W(t)-W(s)$ is independent of $\cF_s$. Below we consider  \textit{cylindrical $Q$-Wiener processes} (see, e.g., \cite[Section 2.5]{PR07}, \cite[Chapters 2-4]{DPZ14}). These can be formally defined as $Q$-Wiener processes in $\dot{H}^{-r}$ for large enough $r>0$, allowing for $\trace(Q) = \infty$. In this case it is no longer true that $W(t)-W(s) \in H$ $P\text{-a.s.}$, but $\inpro{W(t)-W(s)}{f}$ is still a real-valued Gaussian random variable for all $f \in H$ and $\E[\inpro{W(t)-W(s)}{f}\inpro{W(t)-W(s)}{g}] = \inpro{(t-s)Qf}{g}$ for all $f,g \in H$

With this in mind, a predictable process $X = (X(t))_{t \in [0,T]}$, with $T< \infty$ fixed, is called a \textit{mild solution} to~\eqref{eq:spde} if $\sup_{t \in [0,T]} \norm[L^2(\Omega;H)]{X(t)} < \infty$ and for all $t \in [0,T]$,
\begin{equation*}
X(t) = E(t) x_0 + \int^t_0 E(t-s)F(s,X(s)) \dd s + \int^t_0 E(t-s) G(s) \dd W(s), \text{
	$P$-a.s. }
\end{equation*}
The first integral is of Bochner type while the second is an $H$-valued It\^o-integral. For this to be well defined, we need $G$ to map into $\cL^0_2 = \cL_2(Q^{1/2}(H),H)$. Here $Q^{1/2}(H)$ is a Hilbert space equipped with the inner product $\inpro{Q^{-1/2}\cdot}{Q^{-1/2}\cdot}$, where $Q^{-1/2}$ denotes the pesudoinverse of $Q^{1/2}$. We make this explicit in the assumption below, which by \cite[Theorem 2.25]{K14} guarantees the existence of a mild solution to \eqref{eq:spde}. The assumption also implies that the approximation we consider below is Gaussian.

\begin{assumption} The parameters of~\eqref{eq:spde} fulfill the following requirements.
	\label{assumptions:1}
	\hfill
	\begin{enumerate}[label=(\roman*)]
		\item 
		%\label{assumptions:1:Q}
		$W=(W(t))_{t\in[0,T]}$ is an $(\cF_t)_{t \in [0,T]}$-adapted cylindrical $Q$-Wiener process, where the operator ${Q \in \cL(H)}$ is self-adjoint and positive semidefinite, not necessarily of trace class.
		\item %\label{assumptions:1:g} 
		There is a constant $C>0$ such that $G : [0,T] \to \cLnull $ satisfies
		\begin{equation*}
		\norm[\cLnull]{G(t_1)-G(t_2)} \le C |t_1-t_2|^{1/2}, \text{ for all } t_1, t_2 \in [0,T].
		\end{equation*}
		\item %\label{assumptions:1:F} 
		The function $F : [0,T] \times H \to \dot{H}^{-1}$ is affine in $H$, i.e., for each $t \in [0,T]$ there exists an operator $F^1_t \in \cL(H,\dot{H}^{-1})$ and an element $F^2_t \in \dot{H}^{-1}$ such that $F(t, f) = F^1_t f + F^2_t$ for all $f \in H$. Furthermore, there exists a constant $C>0$ such that $F : [0,T] \times H \to \dot{H}^{-1}$ satisfies
		\begin{equation*}
		\norm[-1]{F(t_1, f)-F(t_2, f)} \le C (1+\norm{f}) |t_1-t_2|^{1/2}
		\end{equation*}
		for all $f \in H$, $t_1, t_2 \in [0,T]$, and  $\norm[\cL(H,\dot{H}^{-1})]{F_t^1} \le C$ for all $t \in [0,T]$.
		\item  The initial value $x_0$ is a (possibly degenerate) $\cF_0$-measurable $\dot{H}^1$-valued Gaussian random variable.
	\end{enumerate}
\end{assumption}

By degenerate, we mean that $\Cov(x_0)$ may only be positive semidefinite, allowing for a deterministic initial value. Since $x_0 \in L^p(\Omega;H)$ for all $p \ge 1$ we have $\sup_{t \in [0,T]} \norm[L^p(\Omega;\dot{H}^s)]{X(t)} < \infty$
for all $p \ge 1$ and $s \in [0,1)$.

As a model problem in this context, we consider a stochastic advection-diffusion equation.
\begin{example}
	\label{ex:stoch-conv}
	For a convex polygonal domain $D\subset \R^d$, $d=1,2,3$, let $H=L^2(D)$ and for a function $f$ on $D$, let the operator $-A \colon \dom(-A) \to H$ be given by $A f = \nabla \cdot(a \nabla f)$ with Dirichlet zero boundary conditions, where $a: D \to \R$ is a sufficiently smooth strictly positive function. In this setting it holds (cf.~\cite[Chapter 3]{T06})  that $\dot{H}^1 = H^1_0(D)$ and $\dot{H}^2 = H^2(D) \cap H^1_0(D)$, where $H^k(D)$ is the Sobolev space of order $k$ on $D$ and $H^1_0(D)$ consists of all $ f \in H^1(D)$ such that $f(x) = 0$ for $x \in \partial D$, the boundary of $D$. Let $F$ be given by $F(t,f) = b(t, \cdot) \cdot \nabla f(\cdot) + c(t,\cdot) f(\cdot) + d(t,\cdot)$ for a function $f$ on $D$. When $b: D \times [0,T] \to \R^d$ and $c,d: D \times [0,T] \to \R$ are smooth, $F(t,\cdot)$ is indeed a member of $\cL(H, \dot{H}^{-1})$ (cf. \cite[Example 2.22]{K14}). Choosing $G = g(t)\cdot$, where $g: [0,T]\to\R$ is smooth, $Q$ such that $\trace(Q)<\infty$ and $x_0$ smooth, Equation~\eqref{eq:spde} is interpreted as the problem to find a function-valued stochastic process $X$ such that 
	%	\begin{alignat*}{2}
	%	\dd X(t,x) &= \left( \nabla \cdot(a(x) \nabla X(t,x)) + b(t,x) \cdot \nabla X(t,x) \right) \dd t + \dd W(t,x) &&\text{ for all } t \in (0,T], x \in D, \\
	%	X(t,x) &= 0, &&\text{ for all } t \in (0,T], x \in \partial D, \\
	%	X(0,x) &= x_0(x),   &&\text{ for all } x \in D \setminus \partial D,
	%	\end{alignat*}
	\begin{align*}
	\dd X(t,x) &= \left( \nabla \cdot(a(x) \nabla X(t,x)) + b(t,x) \cdot \nabla X(t,x) + c(t,x) X(t,x) + d(t,x) \right) \dd t \\
	&\quad+ g(t) \dd W(t,x)
	\end{align*}
	for all $t \in (0,T], x \in D$, with $X(t,x)=0$ for all $t \in (0,T], x \in \partial D$ and $X(0,x) = x_0(x)$ for all $x \in D$.	Moreover, the covariance of the noise has a more concrete meaning in this setting. A consequence of $\trace(Q) < \infty$ is, by \cite[Proposition A.7]{PZ07},  the existence of a symmetric square integrable function $q: D \times D \to \R$ such that 
	\begin{equation}
	\label{eq:Q_q}
	Q f = \int_D q(\cdot,y) f(y) \dd y
	\end{equation}
	for $f \in H$. Similarly, if $q$ is a symmetric, positive semidefinite continuous function on $D \times D$, then \eqref{eq:Q_q} defines a covariance operator by \cite[Theorem A.8]{PZ07}. If we take $W(t,\cdot)$, $t \in [0,T]$, to be a random field (which in this setting is to say that it is defined in all of $D$ and jointly $\cB(D) \otimes \cF_t$-measurable, where $\cB(D)$ denotes the Borel $\sigma$-algebra on $D$), then $t q$ is the covariance function of the field.
\end{example}

For the spatial discretization of \eqref{eq:spde}, we assume the setting of \cite[Chapter 3.2]{K14}. Let $(V_h)_{h \in (0,1]}$ be a family of subspaces of $\dot{H}^1$ equipped with $\inpro{\cdot}{\cdot}$ such that $N_h = \dim(V_h) < \infty$. By $P_h: H^{-1} \to V_h$ we denote the generalized orthogonal projector onto $V_h$, defined by $\inpro{P_h f}{\Phi_h} =
{}_{\dot{H}^{-1}}\langle
f, \Phi_h \rangle_{\dot{H}^{1}}$ for all $f \in  \dot{H}^{-1}$ and $\Phi_h \in V_h$, where $_{\dot{H}^{-1}}\langle
\cdot, \cdot \rangle_{\dot{H}^{1}}$ denotes the dual pairing. By $R_h$ we denote the the Ritz projector, i.e., the orthogonal projector $R_h: \dot{H}^1 \to V_h$ with respect to $\inpro[1]{\cdot}{\cdot}$.
We assume that there is a constant $C>0$ such that $\norm[1]{P_h f} \le C \norm[1]{f}$ and $\norm{R_h f - f} \le C \norm[s]{f} h^s $ for all $h \in (0,1]$ and all $f$ in $\dot{H}^1$ and $\dot{H}^s$, $s \in \{1,2\}$, respectively. This setting includes both finite element and spectral methods (see \cite[Example 3.6-3.7]{K14} for details on when these relatively mild assumptions hold).
The operator $-A_h: V_h \to V_h$ is now defined by $\inpro{-A_h f_h}{g_h} = \inpro[1]{f_h}{g_h} = \inpro{(-A)^\frac{1}{2}f_h}{(-A)^\frac{1}{2}g_h}, \text{ for all } f_h, g_h \in V_h$. For the time discretization, we use the drift-implicit Euler method. Let a uniform time grid be given by $t_j = j \Delta t$ for $j = 0, 1, \ldots, N_{\Delta t} = T / \Delta t \in \N$. A fully discrete approximation $(X_{h,\Delta t}^{t_{j}})_{j=0}^{N_{\Delta t}}$ is then given by
\begin{equation}
\label{eq:backward_euler}
X_{h,\Delta t}^{t_{j+1}} - X_{h,\Delta t}^{t_j} = \left(A_h X_{h,\Delta t}^{t_{j+1}} + P_h F(t_j,X_{h,\Delta t}^{t_j})\right) \Delta t + P_h G(t_j) \Delta W^{j},
\end{equation}
where $\Delta W^j = W(t_{j+1}) - W(t_{j})$ and $j = 0, \ldots, N_{\Delta t} - 1$. It converges strongly to the solution of \eqref{eq:spde} in the sense of the following theorem.

\begin{theorem}\cite[Theorem 3.14]{K14}
	\label{thm:strongconvergence}
	Let the terms of \eqref{eq:spde} satisfy Assumption~\ref{assumptions:1} and let $(X_{h,\Delta t}^{T})_{h,\Delta t}$ be a family of approximations of $X(T)$ given by \eqref{eq:backward_euler}. Then, for all $p \ge 1$, $\sup_{h,\Delta t} (\norm[L^p(\Omega;H)]{X_{h,\Delta t}^{T}}) < \infty$ and there is a constant $C > 0$ such that
	\begin{equation*}
	\norm[L^p(\Omega;H)]{X(T)-X_{h,\Delta t}^{T}} \le C\left( h + {\Delta t}^{1/2}\right), \text{ for all } h, \Delta t \in  (0,1].
	\end{equation*}
\end{theorem}

Since the goal of this paper is the approximation of the quantity $\E[\phi(X(T))]$, where $\phi$ is a smooth functional, the concept of weak convergence, i.e., convergence with respect to the expectation of functionals of the solution, is vital as it allows for the efficient tuning of the MC estimators in Sections~\ref{sec:mc} and~\ref{sec:mlmc}. In order to use a result from \cite{K14}, we need a stronger assumption. In particular, $F$ is a function of time only and $\phi$ is smooth. This is formalized below, see also Remark~\ref{rem:weak_convergence}.

\begin{assumption} The parameters of~\eqref{eq:spde} fulfill the following requirements.
	\label{assumptions:2}
	\hfill
	\begin{enumerate}[label=(\roman*)]
		\item %\label{assumptions:1:g} 
		For some $\delta \in [1/2,1]$, there is a constant $C>0$ such that $G : [0,T] \to \cLnull $ and $F : [0,T] \to H$ satisfy
		\begin{equation*}
		\norm[\cLnull]{G(t_1)-G(t_2)} \le C |t_1-t_2|^{\delta}, \text{ for all } t_1, t_2 \in [0,T]
		\end{equation*}
		and
		\begin{equation*}
		\norm{F(t_1)-F(t_2)} \le C  |t_1-t_2|^{\delta}, \text{ for all } t_1, t_2 \in [0,T].
		\end{equation*}
		\item The functional $\phi$ is a member of $C^2_\mathrm{p}(H;\R)$, the space of all continuous mappings from $H$ to $\R$ which are twice continuously Fr\'echet-differentiable with at most polynomially growing derivatives. 
		\item The initial value $x_0 \in \dot{H}^1$ is deterministic.
	\end{enumerate}
\end{assumption}

With this assumption in place, we cite the following weak convergence result.

\begin{theorem}\cite[Theorem 5.12]{K14}
	\label{thm:weakconvergence}
	Under the assumptions of Theorem~\ref{thm:strongconvergence} and  Assumption~\ref{assumptions:2}, there is a constant $C > 0$ such that
	\begin{equation*}
	\left|\E\left[\phi(X(T))-\phi(X_{h,\Delta t}^{T})\right]\right| \le C\left(1 + |\log(h)|\right)\left( h^2 + {\Delta t}^\delta\right), \text{ for all } h, \Delta t \in  (0,1].
	\end{equation*}
\end{theorem}

\section{Covariance-based sampling in a Monte Carlo Setting}
\label{sec:mc}

For $Y \in L^1(\Omega;\R)$, the MC approximation of $\E[Y]$ is given by
\begin{equation*}
E_N\left[Y\right] = \frac{1}{N} \sum_{i=1}^{N} Y^{(i)},
\end{equation*} 
where $N \in \N$ is the number of independent realizations, $Y^{(i)}$, of $Y$. Any $Y \in L^2(\Omega;\R)$ satisfies for $N \in \N$ the inequality 
\begin{equation}
\label{eq:LLN}
\norm[L^2(\Omega;\R)]{\E\left[Y\right]-E_N\left[Y\right]} = \frac{1}{\sqrt{N}} \Var\left(Y\right)^{1/2} \le \frac{1}{\sqrt{N}} \norm[L^2(\Omega;\R)]{Y}.
\end{equation}
In order to accurately approximate $\E[\phi(X(T))]$ by MC using our fully discrete approximation $X^T_{h,\Delta t}$ of $X(T)$, we must therefore generate many samples of $X^T_{h,\Delta t}$. In practice one samples the vector $\bar{\mathbf{x}}_h^{T} = [x_1, x_2, \ldots, x_{N_h}]'$ of coefficients of the expansion $X_{h,\Delta t}^{T} = \sum^{N_h}_{k=1} x_k \Phi^h_k$, where $\mathbf{\Phi}^h = (\Phi^h_k)_{k=1}^{N_h}$ is a basis of $V_h$. 

The classical approach to this is that of \textit{path-based sampling}, i.e., solving the $N_{\Delta t}$ matrix equations corresponding to \eqref{eq:backward_euler} once for each sample $i = 1, 2, \ldots, N$ (Algorithm~\ref{alg:path_mc}). These systems are obtained by expanding~\eqref{eq:backward_euler} on $\mathbf{\Phi}^h$ and applying $\inpro{\Phi^{h}_i}{\cdot}$ to each side of this equality for $i = 1, 2, \ldots, N_h$. 
\begin{algorithm}[h!]
	\begin{algorithmic}[1]
		\STATE $result = 0$
		\FOR{$i=1$ to $N$}
		\STATE Sample increments of a realization $W^{(i)}$ of the $Q$-Wiener process $W$
		\STATE Compute $\bar{\mathbf{x}}_h^{T} = [x_1, x_2, \ldots, x_{N_h}]'$ directly by solving the matrix equations corresponding to the drift-implicit Euler--Maruyama system \eqref{eq:backward_euler} driven by $W^{(i)}$
		\STATE Compute $\phi(X^T_{h,\Delta t}) = \phi\left(\sum^{N_h}_{k=1} x_k \Phi^h_k\right)$
		\STATE $result = result + \phi(X^T_{h,\Delta t})^{(i)}/N$
		\ENDFOR
		\STATE $E_N\left[\phi(X^T_{h,\Delta t})\right] = result$
	\end{algorithmic}
	\caption{Path-based MC method of computing an estimate $E_N[\phi(X^T_{h,\Delta t})]$ of $\E[\phi(X(T))]$ 
	}
	\label{alg:path_mc}
\end{algorithm}

Our alternative approach is that of \textit{covariance-based sampling} where $\Cov(X_{h,\Delta t})$ is computed, yielding the covariance matrix of $\bar{\mathbf{x}}_h^{T}$ which is used to generate samples of $\bar{\mathbf{x}}_h^{T}$ directly (Algorithm~\ref{alg:cov_mc}). This is possible since Assumption~\ref{assumptions:1} ensures that $\bar{\mathbf{x}}_h^{T}$ is Gaussian. To see this, we first introduce the abbreviations $R_{h,{\Delta t}}^{} = (I_H - \Delta t A_h)$, $F^{1,j}_{h,{\Delta t}} = \left(I_H + \Delta t P_h F^1_{t_j} \right)$, and $F^{2,j}_{h,{\Delta t}} = \Delta t P_h F^2_{t_j}$, so that \eqref{eq:backward_euler} can be written as
\begin{equation}
\label{eq:backward_euler_2}
R_{h,{\Delta t}} X_{h,\Delta t}^{t_{j+1}} =  F_{h,\Delta t}^{1,j} X_{h,\Delta t}^{t_j} + F^{2,j}_{h,{\Delta t}} + P_h G(t_j) \Delta W^{j},
\end{equation}
for $j=0,1, \ldots, N_{\Delta t}-1$. 
The $V_h$-valued random variable $X_{h,\Delta t}^{t_{j+1}}$ is Gaussian by induction, due to the facts that affine transformations of Gaussian random variables remain Gaussian, that $X_{h,\Delta t}^{t_j}$ and $P_h G(t_j) \Delta W^{j}$ are independent, that the recursion is started at a (possibly degenerate) Gaussian random variable, and that $P_h G(t_j) \Delta W^{j}$ itself is Gaussian, which can be seen by, for example, \cite[Theorems 4.6, 4.27]{DPZ14}. That $\bar{\mathbf{x}}_h^{T}$ is an $\R^{N_h}$-valued Gaussian random variable is then a consequence of the equality $\inpro[\R^{N_h}]{\bar{\mathbf{x}}_h^{T}}{\mathbf{a}} = \inpro[H]{X_{h,\Delta t}^{T}}{a_h}$, where $\mathbf{a} = [a_1, a_2, \ldots, a_{N_h}]' \in \R^{N_h}$ is arbitrary and $a_h = \sum^{N_h}_{i=1} (\mathbf{M}_h^{-1} \mathbf{a})_i \phi_i$, where $\mathbf{M}_h$ is the symmetric positive definite matrix with entries $m_{i,j} = \inpro{\phi_i}{\phi_j}$. 

In the next theorem we introduce a scheme for the calculation of $\Cov(X_{h,\Delta t})$, inspired by the derivation of stability properties of SPDE approximation schemes in \cite{LPT17}, see also \cite{BS12}. 

\begin{theorem}
	\label{thm:mc_cov}
	Let the terms of \eqref{eq:spde} satisfy Assumption~\ref{assumptions:1} and let $(X_{h,\Delta t}^{t_{j}})_{j=0}^{N_{\Delta t}}$ be given by~\eqref{eq:backward_euler}. Then $\mu^{T} = \E[X_{h,\Delta t}^{{T}}] \in V_h$ and $\Sigma^{T} = \Cov(X_{h,\Delta t}^{{T}}) \in V_h^{\otimes 2}$ are given by the recursions
	\begin{equation}
	\label{eq:mean_recursion}
	R_{h,{\Delta t}}^{} \mu^{t_{j+1}} =  F^{1,j}_{h,{\Delta t}}\mu^{t_j} +  F^{2,j}_{h,{\Delta t}},
	\end{equation}
	\begin{equation}
	\label{eq:cov_recursion}
	\left( R_{h,{\Delta t}}^{} \right)^{\otimes 2} \Sigma^{t_{j+1}} =  \left( F^{1,j}_{h,{\Delta t}} \right)^{\otimes 2} \Sigma^{t_{j}} + \E\left[\left(P_h G(t_j) \Delta W^{j}\right)^{\otimes 2}\right]
	\end{equation}
	for $j=0, 1, \ldots, N_{\Delta t} -1$.
\end{theorem}
\begin{proof}
	We first prove the result assuming $F^2_t=0$ for all $t \in [0,T]$. The recursion scheme \eqref{eq:mean_recursion} for the mean follows by applying $\E[\cdot]$ to both sides of~\eqref{eq:backward_euler_2}, noting that $\E[\cdot]$ commutes with linear operators and that $\Delta W^{j}$ has zero mean. 
	
	For the covariance recursion scheme \eqref{eq:cov_recursion}, we first tensorize \eqref{eq:backward_euler_2} to get 
	\begin{equation}
	\label{eq:tensor_scheme}
	\begin{split}
	\left( R_{h,{\Delta t}}^{} \right)^{\otimes 2} \left(X_{h,\Delta t}^{t_{j+1}}\right)^{\otimes 2} &= \left(  F^{1,j}_{h,{\Delta t}}   \right)^{\otimes 2}\left(X_{h,\Delta t}^{t_{j}}\right)^{\otimes 2} +  F^{1,j}_{h,{\Delta t}} X_{h,\Delta t}^{t_j}  \otimes P_h G(t_j) \Delta W^{j} \\ &\quad\quad+ P_h G(t_j) \Delta W^{j} \otimes  F^{1,j}_{h,{\Delta t}} X_{h,\Delta t}^{t_j}  + \left(P_h G(t_j) \Delta W^{j}\right)^{\otimes 2}.
	\end{split}
	\end{equation}
	Since $\Delta W^j$ is independent of $X_{h,\Delta t}^{t_j}$ and has zero mean,
	\begin{equation*}
	\E\left[F^{1,j}_{h,{\Delta t}} X_{h,\Delta t}^{t_j}  \otimes P_h G(t_j) \Delta W^{j} \right] = \E\left[F^{1,j}_{h,{\Delta t}} X_{h,\Delta t}^{t_j}\right] \otimes \E\left[P_h G(t_j) \Delta W^{j}\right] = 0 
	\end{equation*}
	and similarly, the third term has zero mean.  
	Thus, the mean of \eqref{eq:tensor_scheme} is given by 
	\begin{equation*}
	\begin{split}
	&\left( R_{h,{\Delta t}} \right)^{\otimes 2} \E\left[\left(X_{h,\Delta t}^{t_{j+1}}\right)^{\otimes 2}\right]  = \left( F^{1,j}_{h,{\Delta t}}  \right)^{\otimes 2} \E\left[\left(X_{h,\Delta t}^{t_{j}}\right)^{\otimes 2}\right] + 
	\E\left[\left(P_h G(t_j) \Delta W^{j}\right)^{\otimes 2}\right].
	\end{split}
	\end{equation*}
	Tensorizing \eqref{eq:mean_recursion}, the recursion scheme for the mean, we obtain 
	\begin{equation*}
	\left( R_{h,{\Delta t}}  \right)^{\otimes 2} \left(\mu^{t_{j+1}}\right)^{\otimes 2} =  \left( F^{1,j}_{h,{\Delta t}}  \right)^{\otimes 2} \left(\mu^{t_{j}}\right)^{\otimes 2}
	\end{equation*}
	and by subtracting this from the previous equation we end up with \eqref{eq:cov_recursion}. The general case of a non-zero $F^{2,j}_{h,\Delta t}$ term is proven in the same way, noting that all terms involving this disappears from \eqref{eq:cov_recursion} when subtracting the tensorized mean in the last step.
\end{proof}
\begin{remark}
	Note that this computation can easily be extended to the case of linear multiplicative noise. However, $X_{h,\Delta t}^{{T}}$ is then non-Gaussian, so knowledge of its covariance is not sufficient to compute samples of it.
\end{remark}
\begin{algorithm}[h!]
	\begin{algorithmic}[1]
		\STATE Form the mean vector $\mathbf{\mu}$ and covariance matrix $\mathbf{\Sigma}$ of $\bar{\mathbf{x}}_h^{T}$ by solving the matrix equations corresponding to \eqref{eq:mean_recursion} and \eqref{eq:cov_recursion}
		\STATE $result = 0$
		\FOR{$i=1$ to $N$}
		\STATE Sample $\bar{\mathbf{x}}_h^{T} = [x_1, x_2, \ldots, x_{N_h}]' \sim N(\mathbf{\mu}, \mathbf{\Sigma})$
		\STATE Compute $\phi(X^T_{h,\Delta t}) = \phi\left(\sum^{N_h}_{k=1} x_k \Phi^h_k\right)$
		\STATE $result = result + \phi(X^T_{h,\Delta t})^{(i)}/N$
		\ENDFOR
		\STATE $E_N\left[\phi(X^T_{h,\Delta t})\right] = result$
	\end{algorithmic}
	\caption{Covariance-based MC method of computing an estimate $E_N[\phi(X^T_{h,\Delta t})]$ of $\E[\phi(X(T))]$ 
	}
	\label{alg:cov_mc}
\end{algorithm}

Next, we compare the computational complexities of Algorithms~\ref{alg:path_mc} and~\ref{alg:cov_mc}. 
Combining \eqref{eq:LLN} with Theorem~\ref{thm:weakconvergence} yields, using the triangle inequality, for a constant $C > 0$,
\begin{align*}
&\norm[L^2(\Omega;\R)]{\E[\phi(X(T))] - E_N[\phi(X_{h,\Delta t}^{T})]} \\ 
&\quad \quad \le \norm[L^2(\Omega;\R)]{\E[\phi(X(T))] - \E[\phi(X_{h,\Delta t}^{T})]} + \norm[L^2(\Omega;\R)]{\E[\phi(X_{h,\Delta t}^{T})] - E_N[\phi(X_{h,\Delta t}^{T})]} \\
&\quad \quad \le C \left(\left(1 + |\log(h)| \right) \left(h^2 + {\Delta t}^\delta \right) + N^{-1/2}\right).
\end{align*}
To balance this we couple $\Delta t$ and $N$ by ${\Delta t}^\delta \simeq N^{-1/2} \simeq h^2$. We make the following assumption for the computational complexities of the algorithms.
\begin{assumption}
	\label{assumptions:3}
	There is a $d\in \N$ such that the cost of computing one step of \eqref{eq:backward_euler} is $\Op(h^{-\alpha d})$, where $\alpha \in [1,2]$, while the cost of one step of the tensorized system \eqref{eq:cov_recursion} is $\Op(h^{-2d})$. Moreover, the cost of sampling a Gaussian $V_h$-valued random variable with covariance given by \eqref{eq:cov_recursion} is $\Op(h^{-2d})$. Finally, for Algorithm~i, $i=1,2,3,4$, any additional offline cost is $\Op(h^{-\omega_{i} d})$ for some $\omega_i \in \N$. 
\end{assumption}
\begin{remark}
	In our model problem Example~\ref{ex:stoch-conv}, $d$ is the dimension of the space $\R^d$ in which the domain $D$ is contained and $\dim(V_h)=\Op(h^{-d})$. See Section~\ref{sec:implementation} for a discussion of this assumption. 
\end{remark}

In order to compute an approximation of $\E[\phi(X(T))]$, if we use Algorithm~\ref{alg:path_mc}, we need to solve the drift-implicit Euler--Maruyama system \eqref{eq:backward_euler} $N \cdot N_{\Delta t} = T N \Delta t^{-1}$ times, making the total (online) cost $\Op(N{\Delta t}^{-1}h^{-\alpha d}) = \Op(h^{-4-\alpha d-2/\delta})$. If we use Algorithm~\ref{alg:cov_mc} instead, we need to solve \eqref{eq:cov_recursion} $N_{\Delta t} = T{\Delta t}^{-1}$ times and then sample from the resulting covariance $N$ times, making the total cost $\Op(h^{-2d-2/\delta}) + \Op(h^{-2d-4}) = \Op(h^{-2d-4})$. We collect these observations in the following proposition, which ends this section.
\begin{proposition}
	\label{prop:singlelevel_costs}
	Let $\phi$ and the terms of \eqref{eq:spde} satisfy Assumptions~\ref{assumptions:1} and \ref{assumptions:2}. Assume that $X_{h,\Delta t}^{T}$ is given by $\eqref{eq:backward_euler}$ and that $X(T)$ is the solution to \eqref{eq:spde} at time $T < \infty$. If $N \simeq h^{-4}$ and $\Delta t \simeq h^{2/\delta}$ then there is a constant $C > 0$ such that
	\begin{equation*}
	\norm[L^2(\Omega;\R)]{\E[\phi(X(T))] - E_N[\phi(X_{h,\Delta t}^{T})]} \le C \left(1 + |\log(h)| \right) h^2, \text{ for all } h > 0.
	\end{equation*}
	Under Assumption~\ref{assumptions:3}, the cost of computing $E_N[X_{h,\Delta t}^{T}]$ with Algorithm~\ref{alg:path_mc} is bounded by $\Op(\max(h^{-4-\alpha d-2/\delta},h^{-\omega_1 d}))$ and with Algorithm~\ref{alg:cov_mc} by $\Op(\max(h^{-2d-4},h^{-\omega_2 d}))$.
\end{proposition}

\section{Covariance-based Sampling in a Multilevel Monte Carlo Setting}
\label{sec:mlmc}

For our goal of estimating $\E[\phi(X(T))]$, the MLMC algorithm can be a more efficient alternative to the standard MC algorithm. For a sequence $(Y_\ell)_{\ell \in \N_0}$ of random variables in $L^2(\Omega;\R)$ approximating $Y \in L^2(\Omega;\R)$, where the index $\ell \in \N_0$ is referred to as a level, the \emph{MLMC estimator $E^L[Y_L]$} of $\E[Y_L]$ is, for $L \in \N$, defined by
\begin{equation*}%\label{eq:MLMC-estimator}
E^L[Y_L] = E_{N_0}[Y_0] + \sum_{\ell=1}^L E_{N_\ell}[Y_\ell - Y_{\ell-1}],
\end{equation*}
where $(N_\ell)_{\ell=0}^{L}$ are level specific numbers of samples in the respective MC estimators. 

To apply this algorithm in our setting, we take a sequence $(X^T_\ell)_{\ell \in \N_0}$ of approximations of $X(T)$, given by $X^T_\ell = X^T_{h_\ell, \Delta t_\ell}$, where $(h_\ell)_{\ell \in \N_0}$ is a decreasing sequence of mesh sizes and $\Delta t_\ell^\delta \simeq h_\ell^2$, so that $(\phi(X^T_\ell))_{\ell \in \N_0}$ becomes a sequence approximating $\phi(X(T))$. For notational convenience, we set $\phi(X^T_{-1}) = 0$. Computing $E^L \left[\phi\left(X^T_L\right)\right]$ involves, for each $\ell = 1, 2, \ldots, L$, sampling $\phi\left(X^T_\ell\right) - \phi\left(X^T_{\ell-1}\right)$ $N_\ell$ times (we specify how to choose the sample sizes below). For this it is key that $X^T_\ell$ on the \textit{fine} level $\ell$ and $X^T_{\ell-1}$ on the \textit{coarse} level $\ell - 1$ are positively correlated. In the classical path-based method (Algorithm~\ref{alg:path_mlmc}, see also \cite{A13, BL12a, BLS13, LP17}), this is achieved by computing them on the same discrete realization of $W$, assuming that the family $(V_h)_{h \in (0,1]}$ is nested. 
\begin{algorithm}[h!]
	\begin{algorithmic}[1]
		\STATE $result = 0$
		\FOR{$\ell=0$ to $L$}
		\FOR{$i=1$ to $N_\ell$}
		\STATE Sample increments of a realization $W^{(i)}$ of the $Q$-Wiener process $W$
		\STATE Compute $\bar{\mathbf{x}}_{h_{\ell-1}}^{T} = [x^{\ell-1}_1, x^{\ell-1}_2, \ldots, x^{\ell-1}_{N_{h_{\ell-1}}}]'$ by solving the matrix equations corresponding to \eqref{eq:backward_euler} driven by $W^{(i)}$
		\STATE Compute $\bar{\mathbf{x}}_{h_{\ell}}^{T} = [x^\ell_1, x^\ell_2, \ldots, x^\ell_{N_{h_{\ell}}}]'$ by solving the matrix equations corresponding to \eqref{eq:backward_euler} driven by $W^{(i)}$
		\STATE Compute $\phi(X^T_\ell) - \phi(X^T_{\ell-1}) =  \phi\left(\sum^{N_{h_{\ell}}}_{k=1} x^{\ell}_k \Phi^{h_{\ell}}_k\right) - \phi\left(\sum^{N_{h_{\ell-1}}}_{k=1} x^{\ell-1}_k \Phi^{h_{\ell-1}}_k\right)$
		\STATE $result = result + \left(\phi(X^T_\ell) - \phi(X^T_{\ell-1})\right)/N_\ell$
		\ENDFOR
		\ENDFOR
		\STATE $E^L\left[\phi(X^T_L)\right] = result$
	\end{algorithmic}
	\caption{Path-based MLMC method of computing an estimate $E^L \left[\phi\left(X^T_L\right)\right]$ of $\E[\phi(X(T))]$}
	\label{alg:path_mlmc}
\end{algorithm}

To introduce our alternative covariance-based method, the path-based sampling is rewritten as a system on the space $V_{h'} \oplus V_h$. Consider to this end, for $h, h', \Delta t, \Delta t' \in (0,1]$, a pair $\big((X_{h',\Delta t'}^{t'_{j}})_{j=0}^{N_{\Delta t'}}, (X_{h,\Delta t}^{t_{j}})_{j=0}^{N_{\Delta t}}\big)$ of approximations of $X$, given by the drift-implicit Euler--Maruyama scheme \eqref{eq:backward_euler}. Assume further that they are nested in time, i.e., that $\Delta t' = K \Delta t$ for some $K \in \N$ with $K > 1$. We create an extension $(\hat{X}_{h',\Delta t}^{t_{j}})_{j=0}^{N_{\Delta t}}$ of the coarse approximation $(X_{h',\Delta t'}^{{t'_{j}}})_{j=0}^{N_{\Delta t'}}$ to the finer time grid by $\hat{X}_{h',\Delta t}^{t_{0}} = X_{h',\Delta t'}^{{t'_{0}}}$ and 
\begin{equation*}
\hat{R}^j_{h',{\Delta t}} \hat{X}_{h',\Delta t}^{t_{j+1}} =  \hat{F}_{h',\Delta t}^{1,j} \hat{X}_{h',\Delta t}^{t_j} + \hat{F}^{2,j}_{h',{\Delta t}} + P_{h'} \hat{G}(t_j) \Delta W^{j},
\end{equation*}
for $j=0, 1, \ldots, N_{\Delta t} - 1$, where $\Delta W^j = W(t_{j+1})-W(t_{j})$. The operators are given by 
\begin{equation*}
\hat{R}^j_{h',{\Delta t}} =
\begin{cases}
R_{h',{\Delta t'}}
& \text{if } j+1 = 0 \mod K, \\
I_H & \text{otherwise,} 
\end{cases}
\hspace{3 pt}
\hat{F}_{h',\Delta t}^{1,j}=
\begin{cases}
F_{h',\Delta t'}^{1,j/K}
& \text{if } j+1 = 1 \mod K, \\
I_H & \text{otherwise,} 
\end{cases}
\end{equation*}
\begin{equation*}
\hat{F}_{h',\Delta t}^{2,j} =
\begin{cases}
F_{h',\Delta t'}^{2,j/K}
& \text{if } j+1 = 1 \mod K, \\
0 & \text{otherwise,} 
\end{cases}
\text{ and } \hat{G}(t_j) = G(t_{j-(j\hspace{-5 pt} \mod K)}).
\end{equation*}
Note that $\hat{X}^{t_{j+1}}_{h', \Delta t} = X^{t'_{(j+1)/K}}_{h', \Delta t'}$ when $j+1 = 0 \mod K$ since then
\begin{align*}
\hat{R}^j_{h',{\Delta t}}	\hat{X}^{t_{j+1}}_{h', \Delta t} &= 
\hat{F}_{h',\Delta t}^{1,j} \hat{X}^{t_{j}}_{h', \Delta t} + \hat{F}^{2,j}_{h',{\Delta t}} + P_{h'} \hat{G}(t_j) \Delta W^{j}
= \hat{X}^{t_{j}}_{h', \Delta t} + P_{h'} G(t_{j-(K-1)}) \Delta W^{j} \\
&= \hat{X}^{t_{j-1}}_{h', \Delta t} + P_{h'} G(t_{j-1-(K-2)}) \Delta W^{j-1} + P_{h'} G(t_{j-(K-1)}) \Delta W^{j} = \cdots \\
&= \hat{F}_{h',\Delta t}^{1,j+1-K} \hat{X}_{h',\Delta t}^{t_{j+1-K}} + \hat{F}^{2,j+1-K}_{h',{\Delta t}} + P_h G(t_{j+1-K}) \sum_{i=1}^{K} \Delta W^{j-(i-1)} \\
&= F_{h',\Delta t'}^{1,(j+1)/K-1} \hat{X}_{h',\Delta t}^{t_{j+1-K}} + F^{2,(j+1)/K-1}_{h',{\Delta t'}} + P_h G(t'_{(j+1)/K-1}) \Delta W^{'(j+1)/K-1},
\end{align*}
where $\Delta W^{'(j+1)/K-1} = W(t'_{(j+1)/K})-W(t'_{(j+1)/K-1})$. Hence, sampling the pair of discretizations $\big((X_{h,\Delta t}^{t_{j}})_{j=0}^{N_{\Delta t}}, (X_{h',\Delta t'}^{t'_{j}})_{j=0}^{N_{\Delta t'}}\big)$ on the same realization of the driving $Q$-Wiener process is equivalent to solving the system
\begin{equation*}
\begin{bmatrix}
\hat{R}^j_{h',{\Delta t}} & 0 \\
0 & R^j_{h,{\Delta t}}
\end{bmatrix}
\begin{bmatrix}
\hat{X}_{h', \Delta t}^{t_{j+1}} \\
X_{h, \Delta t}^{t_{j+1}} 
\end{bmatrix}
= 
\begin{bmatrix}
\hat{F}^{1,j}_{h',{\Delta t}} & 0 \\
0 & F^{1,j}_{h,{\Delta t}}
\end{bmatrix}
\begin{bmatrix}
\hat{X}_{h', \Delta t}^{t_{j}} \\
X_{h, \Delta t}^{t_{j}} 
\end{bmatrix}
+
\begin{bmatrix}
\hat{F}^{2,j}_{h',{\Delta t}} \\
F^{2,j}_{h,{\Delta t}}
\end{bmatrix}
+
\begin{bmatrix}
P_{h'} \hat{G}(t_j) \\
P_{h} G(t_j)
\end{bmatrix}
\Delta W^j	
\end{equation*}
in $V_{h'} \oplus V_h$ for $j = 0, 1, \ldots, N_{\Delta t} -1$. We note that $[\hat{X}_{h', \Delta t}^{t_{j}}, X_{h, \Delta t}^{t_{j}}]'$ is a Gaussian $V_{h'} \oplus V_h$-valued random variable for all $j = 0, 1, \ldots, N_{\Delta t} -1$. Therefore, a covariance-based approach for sampling $\big(X_{h,\Delta t}^{T}, X_{h',\Delta t'}^{T}\big)$ could be obtained by directly computing $\Cov\big([ X_{h, \Delta t}^{T}, {X}_{h', \Delta t'}^{T}]'\big) = \Cov\big([ X_{h, \Delta t}^{T}, \hat{X}_{h', \Delta t}^{T}]'\big)$. However, to save computational work, we base  Algorithm~\ref{alg:cov_mlmc} on computing $\Cov\big(X_{h', \Delta t'}^{T}, X_{h, \Delta t}^{T}\big)$ instead. The following theorem gives the scheme for this, which is derived analogously to Theorem~\ref{thm:mc_cov}.
\begin{theorem}
	\label{thm:mlmc_cov}
	Let the terms of \eqref{eq:spde} satisfy Assumption~\ref{assumptions:1} and let, for $h, h', \Delta t, \Delta t' > 0$, $(X_{h,\Delta t}^{t_{j}})_{j=0}^{N_{\Delta t}}$ and  $(X_{h',\Delta t'}^{t'_{j}})_{j=0}^{N_{\Delta t'}}$ be given by the drift-implicit Euler--Maruyama scheme \eqref{eq:backward_euler}. Assume further that $\Delta t' = K \Delta t$ for some $K \in \N$ with $K > 1$. Then the cross-covariance $\Cov\big(X_{h', \Delta t'}^{T}, {X}_{h, \Delta t}^{T}\big)$ is given by $\Sigma_T \in V_{h'} \otimes V_{h}$, where the sequence $(\Sigma_{t_j})_{j=0}^{N_{\Delta t}}$ fulfills 
	\begin{equation}
	\label{eq:cross_cov_recursion}
	\big(\hat{R}_{h',{\Delta t}'}^j \otimes R_{h,{\Delta t}}\big) \Sigma_{t_{j+1}} =  \big(\hat{F}^{1,j}_{h,{\Delta t}}  \otimes F^{1,j}_{h,{\Delta t}} \big) \Sigma_{t_{j}} + \E\left[P_{h'} \hat{G}(t_j) \Delta W^j \otimes P_h G(t_j)\Delta W^j\right].
	\end{equation}
\end{theorem}
\begin{algorithm}[h!]
	\begin{algorithmic}[1]
		\STATE $result = 0$
		\FOR{$\ell=0$ to $L$}
		\STATE Compute the covariance matrix $\mathbf{\Sigma}$ and mean vector $\boldsymbol{\mu}$ of 
		\begin{equation*}
		\bar{\mathbf{x}}_{\ell} =  \left[[\bar{\mathbf{x}}_{h_{\ell-1}}^{T}]^{'},[\bar{\mathbf{x}}_{h_\ell}^{T}]^{'}\right]^{'} = \left[ x^{\ell-1}_1, x^{\ell-1}_2, \ldots, x^{\ell-1}_{N_{h_{\ell-1}}}, x^\ell_1, x^\ell_2, \ldots, x^\ell_{N_{h_{\ell}}} \right]'
		\end{equation*} by computing the means, covariances and cross-covariances of the pair $(X^T_{\ell-1},X^T_\ell)$ via the solution of the matrix equations corresponding to \eqref{eq:mean_recursion}, \eqref{eq:cov_recursion} and \eqref{eq:cross_cov_recursion} 
		\FOR{$i=1$ to $N_\ell$}
		\STATE Sample $\bar{\mathbf{x}}_{\ell} \sim N(\mathbf{\mu},\mathbf{\Sigma})$ 
		\STATE Compute $\phi(X^T_\ell) - \phi(X^T_{\ell-1}) =  \phi\left(\sum^{N_{h_{\ell}}}_{k=1} x^{\ell}_k \Phi^{h_{\ell}}_k\right) - \phi\left(\sum^{N_{h_{\ell-1}}}_{k=1} x^{\ell-1}_k \Phi^{h_{\ell-1}}_k\right)$
		\STATE $result = result + \left(\phi(X^T_\ell) - \phi(X^T_{\ell-1})\right)/N_\ell$
		\ENDFOR
		\ENDFOR
		\STATE $E^L\left[\phi(X^T_L)\right] = result$
	\end{algorithmic}
	\caption{Covariance-based MLMC method of computing an estimate $E^L \left[\phi\left(X^T_L\right)\right]$ of $\E[\phi(X(T))]$}
	\label{alg:cov_mlmc}
\end{algorithm}
The following proposition, which is an adaptation of \cite[Theorem 1]{L16} to our setting, shows how one should choose the sample sizes in an MLMC algorithm and provides bounds on the overall computational work for Algorithms~\ref{alg:path_mlmc} and~\ref{alg:cov_mlmc}.
\begin{proposition}
	\label{prop:multilevel_costs}
	Let $\phi$ and the terms of \eqref{eq:spde} satisfy Assumptions~\ref{assumptions:1} and \ref{assumptions:2}. Let $(h_\ell)_{\ell \in \N_0}$ be a sequence of maximal mesh sizes that satisfy $h_\ell \simeq a^{-\ell}$ for some $a > 1$	and all $\ell \in \N_0$. Let $(X^T_\ell)_{\ell \in \N_0}$ be a sequence of approximations of $X(T)$, where $X^T_\ell = X^T_{h_\ell, \Delta t_\ell}$ is given by the recursion $\eqref{eq:backward_euler}$ with $\Delta t_\ell^\delta \simeq h_\ell^2$.
	
	For $L \in \N, \ell = 1, \ldots, L$, $\varepsilon > 0$, set $N_\ell = \ceil{h_L^{-4} h_\ell^{2} \ell^{1+\varepsilon}}$, where $\ceil{\cdot}$ is the ceiling function, and $N_0 = \ceil{h_L^{-4}}$. Then there exists a constant $C > 0$ such that, for all $L \in  \N$,
	\begin{equation*}
	\norm[L^2(\Omega; \R)]{\E \left[\phi\left(X(T)\right)\right] - E^L \left[\phi\left(X^T_L\right)\right]} \le  C (1 + |\log(h_L)|) h_L^2.
	\end{equation*}
	Under Assumption~\ref{assumptions:3}, the cost of finding $E^L \left[\phi\left(X^T_L\right)\right]$ with Algorithm~\ref{alg:path_mlmc} is bounded by $\Op(\max(h_L^{-2-\alpha d-2/\delta} L^{2+\varepsilon},h_L^{-\omega_3 d}))$. With Algorithm~\ref{alg:cov_mlmc}, the bound of the cost is instead given by $\Op(\max(h_L^{-2d-2/\delta}L, h_L^{-2-2d} L^{2+\varepsilon},h_L^{-\omega_4 d} ))$.
\end{proposition}
\begin{proof}
	By \cite[Lemma 2]{L16},
	\begin{equation}
	\label{eq:mlmc_proof_1}
	\begin{split}
	&\norm[L^2(\Omega; \R)]{\E \left[\phi\left(X(T)\right)\right] - E^L \left[\phi\left(X^T_L\right)\right]} \\ 
	&\quad\le \left|\E\left[\phi\left(X(T)\right) - \phi\left(X^T_L\right)\right]\right| \\ 
	&\quad\quad + \left(N_0^{-1} \norm[L^2(\Omega;\R)]{\phi\left(X^T_0\right)}^2 + \sum_{\ell=1}^{L} N_\ell^{-1} \norm[L^2(\Omega;\R)]{\phi\left(X^T_\ell\right) - \phi\left(X^T_{\ell-1}\right)}^2 \right)^{1/2}.
	\end{split}
	\end{equation}
	By the fact that the Fr\'echet derivative $\phi'$ is at most polynomially growing and by the uniform bound on $X^T_{\ell}$ from Theorem~\ref{thm:strongconvergence}, there is a constant $C>0$ such that $\norm[L^2(\Omega;\R)]{\phi\left(X^T_0\right)}^2 < C$. Moreover, the mean-value theorem for Fr\'echet differentiable mappings (cf. \cite[Example 4.2]{P17}) shows that there exist $p \ge 2$ and $C > 0$ such that 
	\begin{equation*}
	\norm[L^2(\Omega;\R)]{\phi\left(X^T_{\ell}\right) - \phi\left(X(T)\right)}^2 \le C \norm[L^p(\Omega; H)]{X^T_{\ell} - X(T)}^2 \text{ for all } \ell \in \N_0,
	\end{equation*}
	so that, using Theorem~\ref{thm:strongconvergence}, we get a constant $C > 0$ such that
	\begin{align*}
	\norm[L^2(\Omega;\R)]{\phi\left(X^T_\ell\right) - \phi\left(X^T_{\ell-1}\right)}^2 &\le C \left( \norm[L^p(\Omega; H)]{X^T_{\ell} - X(T)}^2 + \norm[L^p(\Omega; H)]{X^T_{\ell-1} - X(T)}^2 \right) \\
	&\le C (h_\ell^2 + h_{\ell-1}^2) \le C (1 + a^{2}) h_\ell^2.
	\end{align*}
	Hence, using Theorem~\ref{thm:weakconvergence} in \eqref{eq:mlmc_proof_1} yields the existence of a constant $C>0$ such that
	\begin{align*}
	&\norm[L^2(\Omega; \R)]{\E \left[\phi\left(X(T)\right)\right] - E^L \left[\phi\left(X^T_L\right)\right]} \\
	&\quad\le C \left((1 + |\log(h_L)|) h_L^2  + \left(N_0^{-1} + \sum_{\ell=1}^{L} N_\ell^{-1} h_\ell^2 \right)^{1/2}\right) \\
	&\quad\le C h_L^2 \left( 1 + |\log(h_L)| + \left(1 + \zeta(1+\varepsilon) \right)^{1/2}\right),
	\end{align*}
	where $\zeta$ denotes the Riemann zeta function and the last inequality follows from the choice of sample sizes. This shows the first part of the theorem. 
	
	If we use Algorithm~\ref{alg:path_mlmc} to compute $E^L[\phi\left(X^T_L\right)]$, the cost of sampling $X^T_0$ $N_0$ times is by Proposition~\ref{prop:singlelevel_costs} $\Op(N_0 h_0^{-\alpha d-2/\delta}) = \Op(h_L^{-4})$. Similarly, since the computation of $\phi\left(X^T_\ell\right) - \phi\left(X^T_{\ell-1}\right)$, $1 \le \ell \le L$, is dominated by the sampling of $X^T_\ell$, the cost for the rest of the terms is bounded by a constant times
	\begin{equation*}
	\sum_{\ell=1}^{L} N_\ell h_\ell^{-\alpha d-2/\delta} = \sum_{\ell=1}^{L} h_L^{-4} h_\ell^{2-\alpha d-2/\delta} \ell^{1+\varepsilon} \le h_L^{-2-\alpha d-2/\delta} L^{2+\varepsilon}. 
	\end{equation*}
	For Algorithm~\ref{alg:cov_mlmc}, the cost of sampling $X^T_0$ $N_0$ times is still $\Op(h_L^{-4})$. For $1 \le \ell \le L$, the cost of computing the covariance of $X^T_\ell$, $X^T_{\ell-1}$ and their cross-covariance is dominated by the cost of computing the covariance of $X^T_\ell$, which is, by the same reasoning as that preceding Proposition~\ref{prop:singlelevel_costs}, $\Op(h_\ell^{-2d-2/\delta})$. The cost of sampling a positively correlated pair of $X^T_\ell$ and  $X^T_{\ell-1}$, given that all covariances have been computed, is $\Op(h_\ell^{-2d}) + \Op(h_{\ell-1}^{-2d}) = \Op(h_\ell^{-2d})$ so the total cost for sampling $\phi\left(X^T_\ell\right) - \phi\left(X^T_{\ell-1}\right)$ for all $\ell = 1, 2, \ldots, L$ is bounded by a constant times 
	\begin{align*}
	\sum_{\ell=1}^{L} \left(h_\ell^{-2d-2/\delta} + N_\ell h_\ell^{-2d} \right) &= \sum_{\ell=1}^{L} \left(h_\ell^{-2d-2/\delta} + h_L^{-4} h_\ell^{2-2d} \ell^{1+\varepsilon} \right) \\ &\le 2 \max(h_L^{-2d-2/\delta}L, h_L^{-2-2d} L^{2+\varepsilon} ).
	\end{align*}
	This finishes the proof.
\end{proof}
\begin{remark}
	We note that this is a suboptimal choice of sample sizes compared to the standard $N_\ell \simeq \sqrt{V_\ell/C_\ell}$, see \cite{G15}, where $V_\ell$ and $C_\ell$ are the variance and computational cost, respectively, of $\phi(X_\ell^T)-\phi(X_{\ell-1}^T)$. The reason for our choice is to avoid the estimation of the additional error resulting from the estimation of $V_\ell$, cf. \cite{BL12a,BLS13}.  
\end{remark}
\begin{remark}
	\label{rem:weak_convergence}
	Note that Assumption~\ref{assumptions:2} is only used to tune the MC and MLMC estimators, using the weak convergence result Theorem~\ref{thm:weakconvergence}. Assuming only Assumption~\ref{assumptions:1}, Theorem~\ref{thm:strongconvergence} can be used for the tuning if $\phi$ is Lipschitz. However, the result can be suboptimal, cf.~\cite{L16}. Moreover, there exist results on weak convergence with different assumptions on the parameters of \eqref{eq:spde} (including rougher noise), for example~\cite{AKL16}. If these are used for tuning, the conclusions regarding which methods are best may be different. The same is true when the parameters of~\eqref{eq:spde} are such that the drift-implicit Euler scheme coincides with a Milstein scheme, see~\cite{BL12}.
\end{remark}

\section{Implementation}
\label{sec:implementation}

In this section we describe the implementation of the algorithms of Sections~\ref{sec:mc} and~\ref{sec:mlmc} in the setting of Example~\ref{ex:stoch-conv}, and motivate Assumption~\ref{assumptions:3} along the way. For simplicity, we conform to Assumption~\ref{assumptions:2} by setting $b=c=0$. For the spatial discretization, let $V_h$ be the space of piecewise linear polynomials on a mesh of $D \subset \R^d, d = 1,2,3$ with maximal mesh size $h>0$, with a basis $\mathbf{\Phi}^h = (\Phi^h_i)_{i=1}^{N_h}$.

Recall that the system of equations in $\R^{N_h}$ corresponding to \eqref{eq:backward_euler} that we must solve in Algorithms~\ref{alg:path_mc} and~\ref{alg:path_mlmc} is given by 
\begin{equation}
\label{eq:singlelevel_matrix_system}
\left(\mathbf{M}_h + \Delta t \mathbf{A}_h \right) \bar{\mathbf{x}}_h^{t_j+1} = \mathbf{M}_h  \bar{\mathbf{x}}_h^{t_j} + \mathbf{f}^{t_j}_h + g(t_j) \mathbf{\Delta} \mathbf{W}_{h}^j, \quad j = 0, 1, \ldots, N_{\Delta t} - 1.
\end{equation}
This gives the vector $\bar{\mathbf{x}}_h^{T}$. Here $\mathbf{M}_h$ is the mass matrix, $\mathbf{A}_h$ the stiffness matrix, $\mathbf{f}^{t_j}_h$ a vector corresponding to the function $d(t_j,\cdot)$ and $(\mathbf{\Delta} \mathbf{W}_{h}^j)^{N_{\Delta t} - 1}_{j=0}$ a family of iid Gaussian $\R^{N_h}$-valued random vectors with covariance matrix $\mathbf{\Sigma}_{h, \Delta W}$, the entries of which are given by $s_{i,j} = \Delta t \int_{D^2} q(x,y) \Phi^h_i(x)  \Phi^h_j(y) \dd x \, \dd y$, $i,j \in \{1, 2, \ldots, N_h\}$. Solving the discretized elliptic problem~\eqref{eq:singlelevel_matrix_system} can be accomplished by direct or iterative solvers. Like the authors of \cite{ABS13,BL12a,BLS13,BSZ11} we assume that we have access to a solver such that the inversion of $\mathbf{M}_h + \Delta t \mathbf{A}_h$ costs $\Op(h^{-d})$ and any additional offline costs are $\Op(h^{-\omega_1 d})$. Whether this is true will depend on the parameters, the mesh and the dimension of the problem. For $d=1$ it is immediately true, for $d>1$ we mention multigrid methods, see, e.g.,  \cite{BS08}. Given this, the cost of solving~\eqref{eq:singlelevel_matrix_system} is dominated by the generation of the stochastic term $\mathbf{\Delta} \mathbf{W}_{h}^j$. In special cases, for example, if $q$ is piecewise analytic (see \cite{KLL10}) or if $Q$ is specified via a truncated Karhunen--Lo\`eve expansion where the truncation does not depend on $h$ (see also Section~\ref{sec:numerics}), the complexity is linear, i.e., $\alpha=1$ in Assumption~\ref{assumptions:3}. If no special assumptions are made on $q$, the standard option is to use a Cholesky or eigenvalue decomposition of the covariance matrix for $P_h \Delta W$ (cf. \cite[Chapter 7]{LPS14}). Since the cost of the decomposition is cubic and the matrix multiplication cost of generating $P_h \Delta W$ is quadratic, we then have $\omega_1, \omega_3 \ge 3$ and $\alpha=2$ in Assumption~\ref{assumptions:3}.

For Algorithms~\ref{alg:cov_mc} and~\ref{alg:cov_mlmc} we must solve tensorized systems. By expanding \eqref{eq:cov_recursion} and applying $\inpro{\Phi^{2,h}_i}{\cdot}$ to each side for $i = 1, 2, \ldots, N_h^2$, where $\mathbf{\Phi}^{2,h} = (\Phi^{2,h}_i)^{N_h^2}_{i=1}$ is a basis of $V_h^{\otimes 2}$, a system of equations in $\R^{2 N_h}$ for the covariance recursion scheme of Theorem~\ref{thm:mc_cov} is obtained. Choosing $\Phi^{2,h}_i = \Phi^h_{\floor{(i-1)/N_h}+1}\otimes \Phi^h_{i-\floor{(i-1)/N_h}N_h}$ for $i=1, 2, \ldots, N_h^2$, the matrices corresponding to $(R_{h,\Delta t}^{})^{\otimes 2}$ and $(F_{h,\Delta t}^{1,j})^{\otimes 2}$ will be Kronecker products of the matrices corresponding to $R_{h,\Delta t}^{}$ and $F_{h,\Delta t}^{1,j}$, $j = 0, 1, \ldots, N_{\Delta t}-1$. In this setting, the resulting system at time $t_{j+1}$ is 
\begin{equation*}
\left(\mathbf{M}_h + \Delta t \mathbf{A}_h \right)^{\otimes_K 2} \bar{\mathbf{y}}_h^{t_j+1} = \mathbf{M}_h^{\otimes_K 2}  \bar{\mathbf{y}}_h^{t_j} + g(t_j)^2 \vect\left({\mathbf{\Sigma}_{h, \Delta W}}\right),
\end{equation*}
where $\otimes_K$ denotes the Kronecker product and $\vect$ the vectorization operator. Here $\bar{\mathbf{y}}_h^{t_j}=\vect(\mathbf{\Sigma}_{\bar{\mathbf{x}}_h^{t_{j}}})$, the covariance matrix of $\bar{\mathbf{x}}_h^{t_{j}}$. By the identity $\vect{(ABC)}=(C'\otimes_K A)\vect(B)$, where $A,B$ and $C$ are matrices such that $ABC$ is well-defined, this is equivalent to the matrix system 
\begin{equation}
\label{eq:singlelevel_cov_matrix_system}
\left(\mathbf{M}_h + \Delta t \mathbf{A}_h \right) 
\mathbf{\Sigma}_{\bar{\mathbf{x}}_h^{t_{j+1}}}
\left(\mathbf{M}_h + \Delta t \mathbf{A}_h \right)
=
\mathbf{M}_h 
\mathbf{\Sigma}_{\bar{\mathbf{x}}_h^{t_{j}}}
\mathbf{M}_h
+ g(t_j)^2 \mathbf{\Sigma}_{h, \Delta W},
\end{equation}
where we have used symmetry of the matrices involved.
Assuming that the solver of~\eqref{eq:singlelevel_matrix_system} is used, the cost of this system is $\Op(h^{-2d})$. To see this, note that since $\mathbf{M}_h$ has $\Op(h^{-d})$ nonzero entries, the right hand side can be formed in $\Op(h^{-2d})$. One then solves for the symmetric matrix $\mathbf{U} = \mathbf{\Sigma}_{\bar{\mathbf{x}}_h^{t_{j+1}}}
\left(\mathbf{M}_h + \Delta t \mathbf{A}_h \right)$ by employing the solver for each of its $\Op(h^{-d})$ columns, and then one similarly solves for the symmetric matrix $\mathbf{\Sigma}_{\bar{\mathbf{x}}_h^{t_{j+1}}}$ in $\left(\mathbf{M}_h + \Delta t \mathbf{A}_h \right)\mathbf{\Sigma}_{\bar{\mathbf{x}}_h^{t_{j+1}}} = \mathbf{U}$. 
Having computed $\mathbf{\Sigma}_{\bar{\mathbf{x}}_h^{T}}$, sampling using this costs $\Op(h^{-2d})$, corresponding to a matrix-vector multiplication, assuming that a Cholesky or eigenvalue decomposition has been used, yielding $\omega_2, \omega_4 \ge 3$. 

In Algorithm~\ref{alg:cov_mlmc}, we also need to solve for the cross-covariance. Given the approximations $(X_{h,\Delta t}^{t_{j}})_{j=0}^{N_{\Delta t}}$ and  $(X_{h',\Delta t'}^{t'_{j}})_{j=0}^{N_{\Delta t'}}$ of Theorem~\ref{thm:mlmc_cov}, choosing a basis $\mathbf{\Phi}^{2,h',h} = (\Phi^{2,h',h}_i)_{i=1}^{N_{h'} N_{h}}$ of $ V_{h'} \otimes V_{h} $ by $\Phi^{2,h',h}_i = \Phi^{h'}_{\floor{(i-1)/N_{h}}+1}\otimes \Phi^{h}_{i-\floor{(i-1)/N_{h}}N_{h}}$ yields, in the same way as above, a matrix system 
\begin{align*}
&\left(\mathbf{M}_{h} + \Delta t \mathbf{A}_{h} \right) 
\mathbf{\Sigma}_{\bar{\mathbf{x}}_{h}^{{t_{j+1}}},\bar{\mathbf{x}}_{h'}^{t_{j+1}}}
\left(\mathbf{M}_{h'} + \Delta t' \mathbf{A}_{h'} \right) \\
&\qquad=
\mathbf{M}_{h} 
\mathbf{\Sigma}_{\bar{\mathbf{x}}_{h}^{{t_{j}}},\bar{\mathbf{x}}_{h'}^{t_{j}}}
\mathbf{M}_{h'}
+ g(t_j)g(t_{j-(j\hspace{-5 pt}\mod K)}) \mathbf{\Sigma}_{{h,h'}, \Delta W},
\end{align*}
corresponding to \eqref{eq:cross_cov_recursion} at time $t_{j+1}$ with $j+1 = 0 \mod K$. Here $\mathbf{\Sigma}_{\bar{\mathbf{x}}_{h}^{{t_{j}}},\bar{\mathbf{x}}_{h'}^{t_{j}}}$ is the cross-covariance matrix of $\bar{\mathbf{x}}_{h}^{t_{j}}$ and $\bar{\mathbf{x}}_{h'}^{t_j}$ while $\mathbf{\Sigma}_{{h,h'}, \Delta W}$ is a matrix with entries given by $s_{i,j} = {\Delta t} \int_{D^2} q(x,y) \Phi^{h}_i(x)  \Phi^{h'}_j(y) \dd x \, \dd y$, $i \in \{1, \ldots, N_{h}\}$, $j \in \{1, \ldots, N_{h'}\}$. By the same reasoning as above, the cost of solving this system is $\Op(\min(h',h)^{-2d})$. Having solved for the cross-covariance matrix, the vector $\bar{\mathbf{x}}^T_{h,h'} =  \left[\big[\bar{\mathbf{x}}_{h'}^{T}\big]',\big[\bar{\mathbf{x}}_{h}^{T}\big]' \right]'$ is sampled at a cost of $\Op(\min(h',h)^{-2d})$ using the matrix
\begin{equation*}
\mathbf{\Sigma}_{\bar{\mathbf{x}}^T_{h,h'}} = \left[
\begin{array}{cc}
\mathbf{\Sigma}_{\bar{\mathbf{x}}_{h'}^{T}} & \mathbf{\Sigma}_{\bar{\mathbf{x}}_{h'}^{T},\bar{\mathbf{x}}_{h}^{T}} \\
\mathbf{\Sigma}_{\bar{\mathbf{x}}_{h'}^{T},\bar{\mathbf{x}}_{h}^{T}}' & \mathbf{\Sigma}_{\bar{\mathbf{x}}_{h}^{T}}
\end{array}
\right],
\end{equation*}
where the diagonal entries are obtained via \eqref{eq:singlelevel_cov_matrix_system} as before.

Having motivated Assumption~\ref{assumptions:3}, we compare the costs of the algorithms. Looking at Propositions~\ref{prop:singlelevel_costs} and~\ref{prop:multilevel_costs}, we see that if $\alpha=2$, Algorithm~\ref{alg:cov_mc} should be used for all $d \in \{1,2,3\}$ in the case that $\delta = 1/2$, while if $\delta = 1$, Algorithm~\ref{alg:cov_mlmc} should be preferred for all $d \in \{1,2,3\}$. If $\alpha=1$, for $d=1,2$ we are in the same case as before, that is to say, Algorithm~\ref{alg:cov_mc} is preferred for $\delta = 1/2$ and Algorithm~\ref{alg:cov_mlmc} for $\delta = 1$. For $d=3$, however, Algorithm~\ref{alg:path_mlmc} outperforms both covariance-based algorithms. Note that we only consider the work needed to solve the problem and not the memory. If the stochastic terms can be generated in linear complexity (i.e. $\alpha=1$), we may not need to store a covariance matrix of size $\Op(h^{-2d})$, which implies that Algorithms~\ref{alg:path_mc} and~\ref{alg:path_mlmc} are preferred when lack of memory is a concern. 

\section{Simulation}
\label{sec:numerics}

In this section we illustrate our result numerically, employing the discretization of the previous section to the stochastic heat equation driven by additive noise,
\begin{equation*}
%\label{eq:stochastic_heat}
\dd X(t) = \Delta X(t) \dd t + \dd W(t),
\end{equation*}
on $H = L^2(D)$ with $D = (0,1)$, for $t \in (0,T] = (0,1]$ with initial value $X(0) = x_0 = x\chi_{(0,1/2)}(x) + (1-x)\chi_{(1/2,1)}(x)$, where $\chi$ denotes the indicator function, and Dirichlet zero boundary conditions. We choose a simple $Q$-Wiener process
\begin{equation*}
W(t,x) = \frac{5\beta_1(t)}{x^{0.45}} + \frac{5\beta_2(t)}{(1-x)^{0.45}},
\end{equation*}
$t \in [0,T]$ and $x \in D$, where $\beta_1, \beta_2$ are two independent standard real-valued Wiener processes on $[0,T]$. This can be generated in linear complexity (i.e., $\alpha=1$ in Assumption~\ref{assumptions:3}). This means that the covariance function corresponding to $Q$ is, for $x,y \in D$, given by
\begin{equation*}
q(x,y) =  \frac{25}{(xy)^{0.45}} + \frac{25}{\left((1-x)(1-y)\right)^{0.45}}.
\end{equation*} A uniform spatial mesh is used in our finite element discretization. 

We now compute approximations of $\E[\phi(X(T))]$, with $\phi(\cdot) = \norm{\cdot}^2 \in C^2_{\mathrm{p}}(H;\R)$. Figure~\ref{fig:mc_conv} shows estimates $\left(\sum_{i=1}^{5} \left(\E[\phi(X(T))] - E_N[\phi(X_{h,\Delta t}^{T})]^{(i)} \right)^2\right)^{1/2},$
for 5 different realizations of $E_N[\phi(X_{h,\Delta t}^{T})]$, of the mean squared errors $\norm[L^2(\Omega;\R)]{\E[\phi(X(T))] - E_N[\phi(X_{h,\Delta t}^{T})]}$
for $h = 2^{-1}, 2^{-2}, \ldots, 2^{-5}$, computed with Algorithms~\ref{alg:path_mc} and~\ref{alg:cov_mc}, choosing $\Delta t$ and $N$ according to Proposition~\ref{prop:singlelevel_costs}. In the case of the covariance-based method, we also include $h= 2^{-6}$ and $h = 2^{-7}$. The quantity $\E[\phi(X(T))]$ is replaced by a reference solution $\E[\phi(X_{h,\Delta t}^{T})]$, with $h = 2^{-8}$, computed with a deterministic method, cf. \cite[Section 6]{P17}. An eigenvalue decomposition was used for the covariance-based method. As expected, the order of convergence is $\Op(h^2)$. In Figure~\ref{fig:mc_costs} we show the computational costs in seconds of the realizations of $E_N[\phi(X_{h,\Delta t}^{T})]$ along with the upper bounds on the costs from Proposition~\ref{prop:singlelevel_costs}. The costs appear to asymptotically follow these bounds. 
\begin{figure}[ht]
	\centering
	\subfigure[Mean square errors for Algorithm~\ref{alg:path_mc} (Path+MC) and Algorithm~\ref{alg:cov_mc} (Cov+MC). \label{fig:mc_conv}]{\includegraphics[width = .49\textwidth]{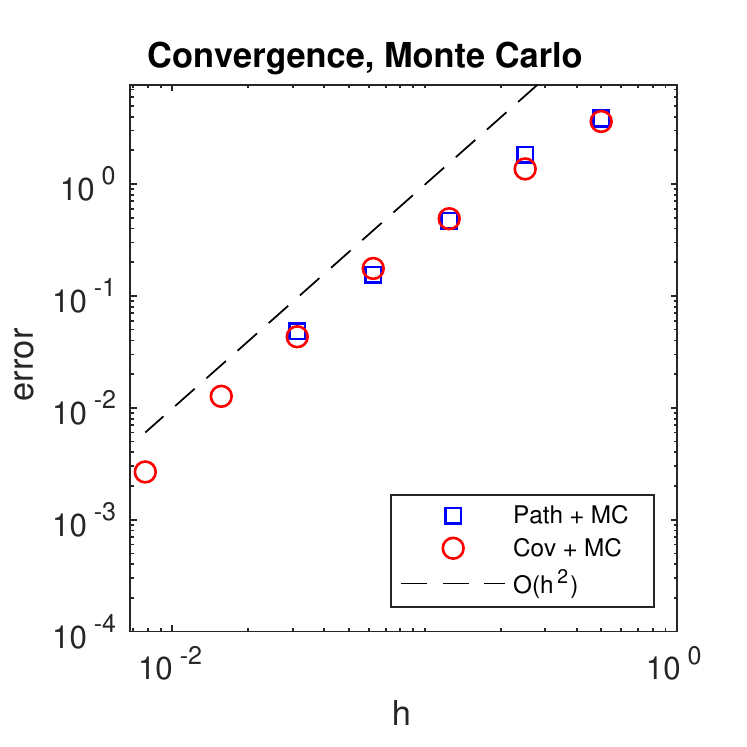}}
	\hspace*{0em}
	\subfigure[Computational costs in seconds for the simulations of Figure~\ref{fig:mc_conv} with bounds from Proposition~\ref{prop:singlelevel_costs}. \label{fig:mc_costs}]{\includegraphics[width = .49\textwidth]{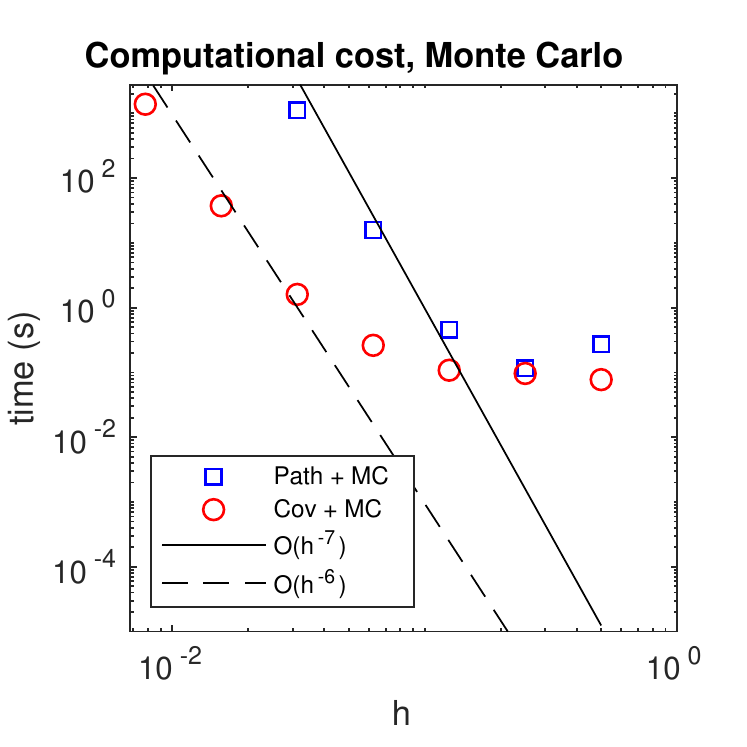}}
	\caption{Convergence and computational costs of Algorithms~\ref{alg:path_mc} and~\ref{alg:cov_mc}.}
\end{figure}

For the MLMC estimator $E^L[\phi(X_L^T)]$ we set, for $\ell = 0, \ldots, L$, $h_\ell = 2^{-\ell-1}$ and choose the temporal step sizes and sample 5 sizes according to Proposition~\ref{prop:multilevel_costs}. In Figure~\ref{fig:mlmc_conv} we show estimates
$\left(\sum_{i=1}^{5} \left(\E[\phi(X(T))] - E^L[\phi(X_L^{T})]^{(i)} \right)^2\right)^{1/2},$
for 5 different realizations of $E^L[\phi(X_L^{T})]$, of the mean squared errors 
$\norm[L^2(\Omega;\R)]{\E[\phi(X(T))] - E^L[\phi(X_L^{T})]}$
for $L = 0, 1, 2, \ldots, 5$ and for Algorithm~\ref{alg:cov_mlmc} also for $L=6$, using the reference solution from before. The order of convergence is again as expected. In Figure~\ref{fig:mlmc_costs} we show the computational costs with the upper bounds on the costs from Proposition~\ref{prop:multilevel_costs}. Both methods appear to  follow the derived complexity bounds.

\begin{figure}[ht]
	\centering
	\subfigure[Mean square errors for Algorithm~\ref{alg:path_mlmc} (Path+MLMC) and Algorithm~\ref{alg:cov_mlmc} (Cov+MLMC). \label{fig:mlmc_conv}]{\includegraphics[width = .49\textwidth]{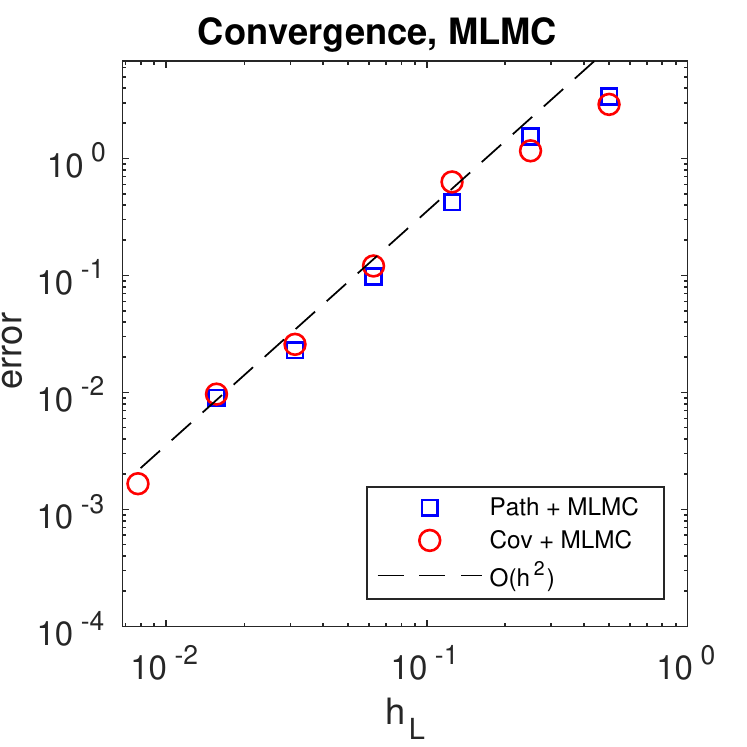}}
	\hspace*{0em}
	\subfigure[Computational costs in seconds for the simulations of Figure~\ref{fig:mlmc_conv} with bounds from Proposition~\ref{prop:multilevel_costs}. \label{fig:mlmc_costs}]{\includegraphics[width = .49\textwidth]{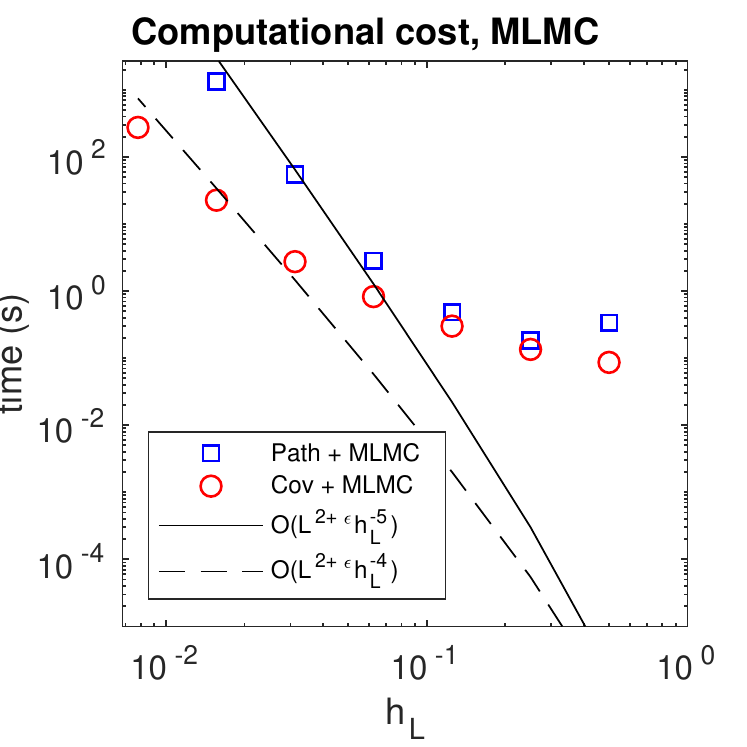}}
	\caption{Convergence and computational costs of the MLMC estimator (Algorithms~\ref{alg:path_mlmc} and~\ref{alg:cov_mlmc}).}
\end{figure}

Finally, for completeness, we show in Figure~\ref{fig:strong_error} a Monte Carlo approximation $ E_N[\norm{X(T) - X_{h,\Delta t}^{T}}^2]^{1/2}$ of the strong error $\norm[L^2(\Omega;H)]{X(T) - X_{h,\Delta t}^{T}}$ for $h = 2^{-1}, 2^{-2}, \ldots, 2^{-6}$, $\Delta t = h^2$, with a reference solution at $h = 2^{-8}$ used in place of $X(T)$. $N = 100$ samples were used. The order of convergence is as expected from Theorem~\ref{thm:strongconvergence}.

The computations in this section were performed in MATLAB\textsuperscript{\textregistered} R2017b on a laptop with a dual-core Intel\textsuperscript{\textregistered} Core\textsuperscript{TM} i7-5600U 2.60GHz CPU. 

\begin{figure}[ht]
	\centering
	{\includegraphics[width = .49\textwidth]{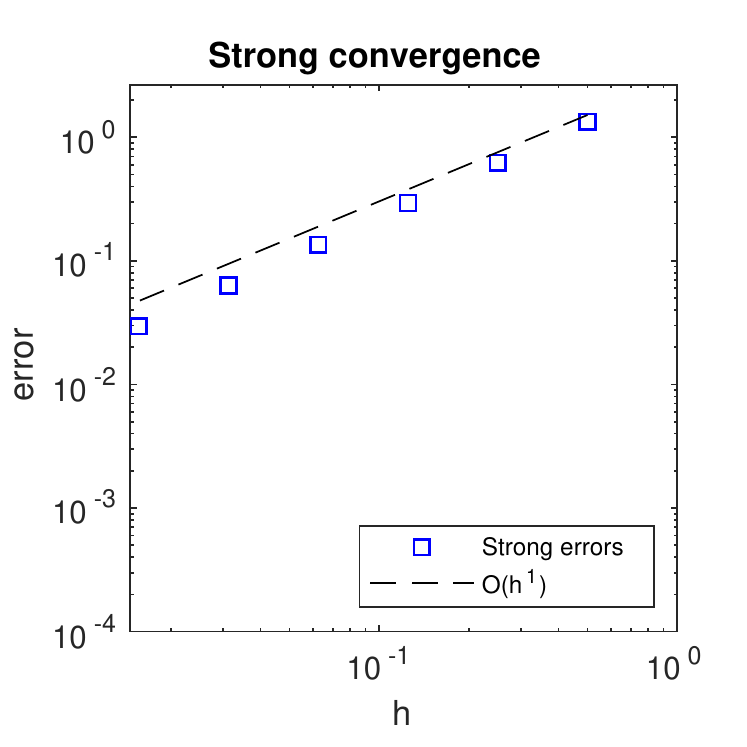}}
	\caption{A Monte Carlo estimate of the strong error $\norm[L^2(\Omega;H)]{X(T) - X_{h,\Delta t}^{T}}$.} \label{fig:strong_error}
\end{figure}
\bibliographystyle{hplain}
\bibliography{cov-sampling}

\end{document}